\documentclass[10pt]{amsart}
\usepackage{amsmath}
\usepackage{graphicx}
\usepackage{amsfonts}
\usepackage{amssymb}
\usepackage{color}
\usepackage{epsfig}
\usepackage{url}
\usepackage{mathrsfs,a4wide}
\usepackage{calligra}
\usepackage{graphicx}
\usepackage{amsthm}
\usepackage{thmtools}
\usepackage{enumitem}
\usepackage{hyperref, cleveref}
\usepackage{mathrsfs}
\usepackage{tikz}

\theoremstyle{plain}

\theoremstyle{definition}
\theoremstyle{remark}
\numberwithin{equation}{section}

\theoremstyle{plain} \declaretheorem[numberwithin = section, name = Theorem,
 refname = {Theorem}, Refname = {Theorem}]{thm}
\theoremstyle{plain} \declaretheorem[numberlike = thm, name = Proposition,
 refname = {Proposition}, Refname = {Proposition}]{prop}
\theoremstyle{plain} \declaretheorem[numberlike = thm, name = Lemma,
refname = {Lemma}, Refname = {Lemma}]{lem}
\theoremstyle{definition} \declaretheorem[numberlike = thm, name = Definition,
 refname = {Definition}, Refname = {Definition}]{df}
\theoremstyle{definition} 
\theoremstyle{definition} \declaretheorem[numberlike = thm, name = Remark,
refname = {Remark}, Refname = {Remark}]{rem}
\theoremstyle{plain} 

\DeclareMathOperator {\real}{\mathbb{R}}
\DeclareMathOperator {\R}{\mathbb{R}}
\DeclareMathOperator {\N}{\mathbb{N}}
\DeclareMathOperator {\BV}{BV}
\DeclareMathOperator{\weak*}{\begin{picture}(10,4)
							\put(0,-2){$\rightharpoonup$}
							\put(3,3){$\ast$}
							\end{picture}}
\DeclareMathOperator {\HN-1}{\mathcal{H}^{N-1}}
\DeclareMathOperator{\gammalimsup}{\Gamma\text{-}\limsup}
\DeclareMathOperator{\gammaliminf}{\Gamma\text{-}\liminf}
\DeclareMathOperator{\gammalim}{\Gamma\text{-}\lim}

\newcommand{\e}{\varepsilon}

\title[Surfactants in a non-local model for phase transitions]
{Surfactants in a non-local model for phase transitions}

\author[M. Cicalese]
{M. Cicalese}
\address[Marco Cicalese]{
	Zentrum Mathematik - M7, Technische Universitat M\"unchen, Boltzmannstrasse 3, 85748 Garching, Germany	
}
\email[M. Cicalese]{cicalese@ma.tum.de}

\author[T. Heilmann]
{T. Heilmann}
\address[Tim Heilmann]{
	Zentrum Mathematik - M7, Technische Universitat M\"unchen, Boltzmannstrasse 3, 85748 Garching, Germany	
}
\email[T. Heilmann]{heilmant@ma.tum.de}

\begin{document}
	
	\begin{abstract}

We investigate the influence of surfactants on stabilizing the formation of interfaces in non-local anisotropic
two-phase fluid at equilibrium. The analysis focuses on singularly perturbed non-local van der 
Waals–Cahn–Hillard-type energies, supplemented with a term that accounts for the interaction between the
surfactant and the fluid. We derive by $\Gamma$-convergence the effective surface tension model as the
thickness of the transition layer vanishes and show that it decreases when the surfactant segregates to 
the interface. 
\end{abstract}

\maketitle
\tableofcontents

\section*{Introduction}

Within the framework of the gradient theory of fluid-fluid phase transitions, surfactants (surface active agents)
 can be incorporated into the model following the ideas proposed by Perkins, Sekerka, Warren and Langer. 
The resulting energy functional, which represents one of the easiest possible toy models for foam stability, 
is obtained as a modification of the classical van der Waals-Cahn-Hillard energy. The latter is a phase-field
approximation of the classical perimeter functional and it has been analyzed by Modica and Mortola in terms
of $\Gamma$-convergence in the limit as the transition length vanishes in \cite{mod, modmor}. 
Then same kind of analysis has been carried out in \cite{fms} for the surfactant model. In particular an
integral term accounting for the fluid-surfactant interaction is added to the classical Cahn-Hillard functional
to model the adsorption of surface active molecules onto phase interfaces. As a result, the effective surface
tension energy of a phase transition decreases as the density of the surfactant at the phase interface
increases. The possibility of modulating phase transitions by changing the distribution of surfactant has found
many practical applications in enhanced oil recovery, emulsion formation or foam stability to cite a few
(see for instance \cite{surfactantbook} and the references therein).\\

In this paper, we revisit the aforementioned theory and the mathematical results obtained in \cite{fms}
(see also \cite{acbo, bbm} for extensions to more general models of fluid-fluid or multiphase-fluid-fluid 
phase transitions in the presence of surfactants, and \cite{ch} for a surfactant model applied to solid-solid
phase transitions) when the local model is replaced by a non-local one. By this, we mean that spatial
inhomogeneities in the order parameter describing the fluid phases are captured through long-range
differential operators, rather than traditional gradients. Energy functionals as the one we consider here arise
for instance as continuum limits of Ising spin systems with Kac potentials in equilibrium statistical mechanics
(see e.g. \cite{ABCP}). \\

It is noteworthy that our research aligns with a growing interest in non-local variational problems, which has
emerged over the past two decades. For an overview of this field, we direct the reader to the survey
\cite{DNPV} and the book \cite{BV}. Additionally, for a general variational framework dealing with
convolution-type energies in Sobolev spaces, we highlight the recent work \cite{AABPT}. 
From the perspective of applied analysis, significant interest has also developed around the theory of
peridynamics, which has seen various applications in this context
(see, for example, \cite{BFGS, DLWT, DZ, SL}).\\

Given a smooth bounded open set $\Omega \subset \real^N$ (the region occupied by the fluid and the
 surfactant), one considers a scalar function $u:\Omega\to \real$ and a non-negative function 
$\rho:\Omega\to [0,+\infty)$ representing the order parameter of the fluid and the density of the surfactant.
For $u,\rho \in L^1(\Omega)$ we set $\mu=\rho\,dx$ and define for $\e>0$ the family of energy functionals
\begin{align*}
	F_\varepsilon(u,\mu) :=&  \frac{1}{\varepsilon} \int_\Omega W(u(x)) dx + \varepsilon \int_\Omega \int_\Omega
		\frac{1}{\e^N}J\left(\frac{y-x}{\e}\right) \left( \frac{|u(y) - u(x)|}{\varepsilon} \right)^2 dy \,dx \\
	&+\varepsilon \int_\Omega \left( \int_\Omega 
		\frac{1}{\e^N}J\left(\frac{y-x}{\e}\right)  \frac{|u(y) - u(x)|}{\varepsilon} \,dy - \rho(x)  \right)^2 dx.
\end{align*}
Assuming $W:\real\to[0,+\infty)$ to be a double-well potential with wells at $-1$ and $+1$ like
$W(z)=(z^2-1)^2$ and the kernel $J\in L^1(\real^N, [0,\infty))$ to be even and such that 
$\int_{\real^N} J(h)|h| \, dh < \infty$ (see Section \ref{assumption} for the precise assumptions on $W$ and
$J$), one is interested  in the asymptotic behaviour as $\varepsilon$ tends to zero of $F_\varepsilon(u,\mu)$
in the sense of $\Gamma$-convergence. The first two integral terms in $F_\e(u,\mu)$ define the functional
$NCH_\e$, a non-local version of the classical Cahn-Hillard functional, namely
\begin{align*}
	NCH_\varepsilon(u) :=&  \frac{1}{\varepsilon} \int_\Omega W(u(x)) dx + \varepsilon \int_\Omega \int_\Omega
		\frac{1}{\e^N}J\left(\frac{y-x}{\e}\right) \left( \frac{|u(y) - u(x)|}{\varepsilon} \right)^2 dy \,dx. 
\end{align*}
The variational analysis of $NCH_\e$ as $\e$ vanishes has been carried out in \cite{ab} (see also \cite{ab2}
for fine properties of the minimizers). In particular in \cite{ab} it has been proved the pre-compactness in
$BV(\Omega;\{-1,+1\})$ of sequences of phase-fields $u_\e$ with uniformly bounded $NCH_\varepsilon$
energy and it has been computed the $\Gamma$-limit of $NCH_\e$ as $\e$ tends to zero with respect to the
$L^1$ convergence. Roughly speaking, the limit $u$ of a converging subsequence of energy bounded 
$u_\e$ is forced to take the values $-1$ and $1$ almost everywhere, partitioning $\Omega$ in the two 
sets $\{u=-1\}$ and $\{u=1\}$. These two sets are interpreted as the pure phases of the fluid while their
common boundary, corresponding to the jump set $S_u$ of the function $u$, as the phase interface. The
effective asymptotic energy of the system, captured by the $\Gamma$-limit of $NCH_\e$, is proved to be
proportional to the surface measure of the interface $S_u$ weighted by the anisotropic surface tension 
$\sigma$ which depends on the shape of $W$. If one fixes the measure of the set $\{u_\e=-1\}=m$ to be
strictly smaller than the measure of $\Omega$, both phases will be non empty. Moreover as $\e$ vanishes 
the minimal non-local Cahn-Hillard energy corresponds to the partition of $\Omega$ in two sets 
$\{u=-1\}$ (having measure $m$) and $\{u=1\}$ having the least anisotropic perimeter of the common
boundary. Such an energy is achieved along a sequence $u_\e$ of phase fields whose non-local spatial
inhomogeneity $ \int_\Omega J_\varepsilon(y-x)  \frac{|u_\e(y) - u_\e(x)|}{\varepsilon} \,dy$ concentrates on
$x\in S_u$. In this perspective, one can understand the role of the third term in $F_\varepsilon$ as modelling
the interaction between the surfactant and the fluid and favouring the phases of minimizers of $NCH_\e$ to
separate where the surfactant is present. The main result of this paper is stated in Theorem
\ref{thm:Gammalim} and it is obtained combining Theorem \ref{liminf} and Theorem \ref{limsup}. 
It shows that carrying out the $\Gamma$-limit of (an extension of) $F_\e$ with respect to the strong 
$L^1$ convergence of the phase fields and the weak$\ast$-convergence of the surfactant measures, one
obtains a limit functional which is finite for $u\in BV(\Omega, \{ -1,1 \})$ and 
$\mu\in{\mathcal M}(\Omega)$ (the space of positive Radon measures) where it takes the form
\begin{align}\label{introlimit}
	F(u,\mu) := \int_{S_{u}} \sigma \left(\nu_u, \frac{d \mu}{d \HN-1 \llcorner S_{u}} \right)\, d \HN-1 . 
\end{align}
In the formula above $\nu_u$ represents the measure-theoretic normal to the phase interface, 
$S_u$ denotes the jump set of $u$, and 
$\sigma: S^{N-1}\times[0,+\infty)\to[0,+\infty)$ is the anisotropic surface tension per unit length and 
it is obtained by
the blow-up formula in Definition \ref{temp7}. For each direction $\nu \in S^{N-1}$, the function 
$\sigma(\nu, \cdot)$ decreases as a function of the relative density of the surfactant measure with respect to
the surface measure of the interface. Consequently, increasing the surfactant density at the interface lowers
the energy cost of the phase boundary, a phenomenon characteristic of immiscible fluids in the presence of
surface-active agents. It is important to note that the limit energy in \eqref{introlimit} differs from the
isotropic one derived in \cite{fms}, where the surfactant model was obtained within the classical van der
Waals-Cahn-Hilliard gradient theory of phase transitions. Additional examples of anisotropic surface tension
energies can be found in \cite{abcs, acs}, where these are obtained as discrete-to-continuum 
$\Gamma$-limits, modeling the coarse-graining process originating from the microscopic 
Blume-Emery-Griffiths ternary surfactant model. In that model, the anisotropy arises due to the underlying
lattice structure, as in classical Ising models, where particle interactions occur only along the lattice directions. 

\section{Notation and preliminaries}
\subsection{Notation}We denote by ${\mathcal L} ^N$ and ${\mathcal H}^{N-1}$ 
the $N$-dimensional Lebesgue 
measure and the $(N-1)$-dimensional Hausdorff measure in $\mathbb R^N$, respectively.
Given  a Lebesgue measurable set $A \subset \real^N$ we denote its Lebesgue measure by $\mathcal{L
}^N(A)$ or by $|A|$. We denote the open ball centred at $x$ with radius $r$ in $\real^N$ as $B_r(x)$ or
 by $B(x,r)$ and set $\omega_N=|B_1(0)|$.
In the whole paper $\Omega \subset \real^N$ is assumed to be a bounded open set with regular boundary.
We  write $\mathcal{M}(\Omega)$ for the space of positive finite Radon measures on $\Omega$, which
we endow with the weak$\ast$ convergence.
Given vectors $x, e \in \real^N$ where $e \neq 0$, we write $x_e$ for the component of $x$ in direction $e$.
Given a function $u \in L^1(\Omega, \real)$, we denote by $S_u$ the approximate discontinuity set of $u$,
i.e., the set of those points $x \in \Omega$ for which no $z \in \real$ exists such that
$\lim_{r \rightarrow 0^+} |B_r(x)|^{-1} \int_{B_r(x)} |u(y) -z| dy = 0$ holds.
We denote by $BV(\Omega)$ the set of functions of bounded variation in $\Omega$. 
We say that a measurable set $E\subset \real^N$ is a set of finite perimeter in 
$\Omega$ if $\chi_E\in BV(\Omega)$. Denoting by $P(E,\Omega)$ the De Giorgi's perimeter of $E$ in 
$\Omega$, if $E$ is a set of finite perimeter we also write that
$P(E,\Omega)={\mathcal H}^{N-1}(\partial^*E\cap\Omega)<+\infty$ where $\partial^*E$ 
stands for the reduced boundary of $E$.
If $u\in BV(\Omega;\{a,b\})$ is a function of bounded variation in $\Omega$ taking only the 
two distinct values $a,b\in\real$, the $(N-1)$-Hausdorff measure
of $S_u$ equals the perimeter of the level set $\{u = a\}$ (and $\{u = b\}$) in $\Omega$ or in formula 
$\HN-1(S_u) = P(\{u = a\}, \Omega)$. In what follows we will denote by $\nu_u$ the measure theoretic
 normal to $S_u$. For all properties of functions of bounded variations and of 
sets of finite perimeter needed in this paper we refer the reader to \cite{afp}. Finally 
we denote by $c$ and $C$ generic real
 positive constants that may vary from line to line and expression to expression within the same formula.

\subsection{The model}\label{assumption}
In this section we introduce the energy functional whose asymptotic variational behaviour we are going
to analyze in the rest of the paper.
We consider potential functions $W$ that satisfy the following set of assumptions
\begin{align}\label{potential}
	&W:\real \rightarrow [0,\infty) \text{ is continuous,} \\
	&W \text{ has exactly two zeroes at the points $-1$ and $1$}, \nonumber \\
	&W \text{ has at least linear growth at inifinity} \nonumber
\end{align}
and kernels $J$ such that
\begin{align}\label{kernel}
	&J \in L^1(\real^N, [0,\infty)), \\
	&J \text{ is even, i.e. } J(-h) = J(h), h \in \real^N \nonumber \\
	&\int_{\real^N} J(h)|h| \, dh < \infty \nonumber.
\end{align} 
In what follows we set $J_\varepsilon(h) := \frac{1}{\varepsilon^N} J(h/\varepsilon)$.
\begin{df} \label{defenergy}
Given a measurable set  $A \subset \real^N,\, u \in L^1(A)$ and $\rho \in L^1(A,[0,\infty))$ we define
\begin{align*}
	F_\varepsilon(u,\rho,A) :=&  \frac{1}{\varepsilon} \int_A W(u(x)) dx + \varepsilon \int_A \int_A
		J_\varepsilon(y-x) \left( \frac{|u(y) - u(x)|}{\varepsilon} \right)^2 dy \,dx \\
	&+\varepsilon \int_A \left( \int_A 
		J_\varepsilon(y-x)  \frac{|u(y) - u(x)|}{\varepsilon} \,dy - \rho(x)  \right)^2 dx.
\end{align*}
For $u \in L^1(\real^N)$ and $\rho \in L^1(A,[0,\infty))$ we define
\begin{align*}
	\mathcal{F}_\varepsilon(u,\rho,A) :=&  \frac{1}{\varepsilon} \int_A W(u(x)) dx +
		 \varepsilon \int_A \int_{\real^N}
		J_\varepsilon(y-x) \left( \frac{|u(y) - u(x)|}{\varepsilon} \right)^2 dy \, dx \\
	&+\varepsilon \int_A \left( \int_{\real^N}
		J_\varepsilon(y-x)  \frac{|u(y) - u(x)|}{\varepsilon} dy - \rho(x)  \right)^2 dx
\end{align*}
and
\begin{align*}
	\mathcal{F}(u,\rho,A) :=&  \int_A W(u(x)) dx + \int_A \int_{\real^N}
		J(y-x) \left( |u(y) - u(x)| \right)^2 dy \, dx \\
	& +\int_A \left( \int_{\real^N}
		J(y-x)  |u(y) - u(x)| dy - \rho(x)  \right)^2 dx
\end{align*}
\end{df}
In order to state the $\Gamma$-convergence result later on we extend the definition of $F_\varepsilon$
to $L^1(\Omega) \times \mathcal{M}(\Omega)$.
\begin{df} \label{defenergyextended}
Given a measurable set $A \subset \real^N$ we define 
$F_\varepsilon(\cdot,\cdot,A): L^1(A) \times  \mathcal{M}(A) \rightarrow [0,\infty]$ as
\begin{equation}
F_\varepsilon(u,\mu,A)=\begin{cases}
F_\varepsilon(u,\rho,A),&\,\text{if }\mu=\rho{\mathcal L}^N\\+\infty,&\, \text{otherwise}.
\end{cases}
\end{equation}
\end{df}

Here below we define the effective energy density of our limit model. The latter will depend on the density 
of the limit surfactant measure on the jump set $S_u$ of the order parameter $u$ and on the normal to $S_u$. 

\begin{df} \label{temp7}
Given $e \in \real^N$ we denote by $\mathcal{C}_e$, the set of all $(N-1)$-dimensional cubes 
centered at $0$ which lie on the hyperplane orthogonal to $e$.
For $C \in \mathcal{C}_e$ we say that $u: \real^N \rightarrow \real$ is $C$-periodic, if
\begin{equation*}
	u(x + ra) = u(x),
\end{equation*}
where $a$ is any of the $N-1$ axes of $C$ and $r$ denotes the side length of $C$.
For any $C \in \mathcal{C}_e$ we also set
\begin{equation*}
T_C := \{y + t e : y \in C, t \in \real\}.
\end{equation*}
Given $C \in \mathcal{C}_e$ and $\gamma \geq 0$ we define the set of admissible $(u,\rho)$ pairs  
$\mathcal{A}(e,\gamma)$ as
\begin{align*}
	\mathcal{A}(e,\gamma) := &\Big\{(u,\rho) |\, u: \real^N \rightarrow [-1,1] 
		\text{ is } C\text{-periodic, } \lim_{x_e \rightarrow \pm \infty} u(x) = \pm 1, \\
	&\rho: \real^N \rightarrow [0,\infty)\text{ is } C\text{-periodic, } \rho \in L^1(T_C) ,
		\int_{T_C} \rho \, dx \leq  \gamma\HN-1(C) \Big\}
\end{align*}
and the limit energy density $\sigma(e, \gamma)$ as
\begin{equation*}
	\sigma(e, \gamma) := \inf \{ \mathcal{F}(u,\rho,T_C) (\HN-1(C))^{-1} : 
		(u,\rho) \in \mathcal{A}(e,\gamma) \}.
\end{equation*}
\end{df}

\begin{rem}\label{rm-sigma-decr}
We notice that for $\gamma_1\leq\gamma_2$ it holds that
$\mathcal{A}(e,\gamma_1)\subset \mathcal{A}(e,\gamma_2)$. Hence by the very definition of 
$\sigma(e,\gamma)$, we obtain that $\sigma(e,\gamma_1)\geq\sigma(e,\gamma_2)$.
\end{rem}

The effective limit energy of our model
$F(u,\mu,\Omega):L^1(\Omega) \times \mathcal{M}(\Omega)\to[0,+\infty]$ is then defined as 
\begin{equation} \label{limitenergy}
F(u,\mu,\Omega) := 
\begin{cases}
    \displaystyle \int_{S_u} \sigma \left( \nu_{u}, \frac{d \mu}{d \HN-1 \llcorner S_u} \right) d\HN-1, & 
	\text{if } u \in \BV(\Omega, \{-1,1\}) \\
    +\infty, & \text{otherwise},
\end{cases}
\end{equation}
where $\sigma$ is as in \cref{temp7}.

\subsection{Preliminary results}
In this section we collect some preliminary definitions and results from \cite{ab}.

We start with the definition of polyhedral functions.
\begin{df} \label{D5.1}
\emph{Polyhedral sets} $P \subset  \real^N$ are defined as open sets with a Lipschitz boundary
that is contained in the union of finitely many affine hyperplanes in $\real^N$.
The intersection of any such hyperplane with the boundary of $P$ is called a \emph{face} of $P$. 
A \emph{polyhedral set in $\Omega$} is given as the intersection of a polyhedral 
set in $\R^N$ wit $\Omega$.
A function $u \in \BV(\Omega, \{-1,1\})$ is said to be a \emph{polyhedral function}
corresponding to a polyhedral set $A \subset \real^N$, if $\HN-1 (\partial A \cap \partial \Omega) = 0$
and $\{u = 1\} = \Omega \cap A$ (and $\{u = -1\} = \Omega \setminus A$).
A subset $\Sigma$ of an affine hyperplane $H$ in $\R^N$ that is a polyhedral set in $\R^{N-1}$ 
is said to be a \emph{$(N-1)$-dimensional polyhedral set} in $\R^N$.
Let $E\subset \R^N$ and let $\Sigma$ be a $(N-1)$-dimensional polyhedral set in $\R^N$. 
Denoting by $\nu_\Sigma$ the normal to the hyperplane $H$, we define the \emph{projection}
of $\Sigma$ to $E$ as
the set $E_\Sigma := \{ x  \in E \,|\, \exists \, t \in \R, y \in \Sigma  : x = y + t \nu_\Sigma \}$.
\end{df}

We continue by defining a notion of convergence of a non-local notion of traces of a sequence of functions.
\begin{df}
Let $A \subset \real^N$ be an open set, $\Sigma \subset \real^N$ a Lipschitz hypersurface and
$v:\Sigma \rightarrow [-1,1]$. For $h\in\mathbb{N}$ let 
$(\e_h)_h$ be such that $\varepsilon_h \rightarrow 0$ and let $(u_h)_h$ be a sequence of functions 
$u_h : A \rightarrow [-1,1]$. We say that the \emph{ $\varepsilon_h$-traces of $u_h$ relative to $A$ 
converge on $\Sigma$} to $v$, if 
\begin{equation*}
	\int_\Sigma \int_{ \{ \frac{1}{\varepsilon_h} (A - y) \} } \left(
	\int_0^1 J\Big(\frac{x}{t}\Big)\frac{|x|}{t}t^{-N}\,dt |u_h(y + \varepsilon_h x) -v(y)| \right)dx \,d\HN-1(y)
	 \rightarrow_{h \rightarrow \infty} 0.
\end {equation*}
\end{df}

The next two propositions are proven in \cite[Lemma 2.7]{ab} and \cite[Proposition 2.5]{ab}, respectively.
\begin{prop} \label{L2.7}
Let $A, A' \subset \real^N$ be disjoint sets and let $\Sigma \subset \real^N$ be a Lipschitz hypersurface
such that for any $x \in A, x' \in A'$, the segment $[x,x']$ intersects $\Sigma$. Let moreover
sequences $\varepsilon_h \rightarrow 0^+$ and $u_h: A \cup A' \rightarrow [-1,1]$ be given.
If the $\varepsilon_h$-traces of $u_h$ relative to $A$ and $A'$ converge on $\Sigma$ to
$v: \Sigma \rightarrow [-1,1]$ and $v': \Sigma \rightarrow [-1,1]$ respectively, 
then there exists a constant $C$ depending only on $J$ such that
\begin{equation*}
	\limsup_{h \rightarrow \infty} \frac{1}{\varepsilon_h} \int_A \int_{A'} J_{\varepsilon_h}(y - x)
	( u_h(y) - u_h(x) )^2 \, dy \, dx  \leq C \int_\Sigma |v'(x) - v(x)| \, d\HN-1(x).
\end{equation*}
\end{prop}

\begin{rem}\label{rm-linear-traces}{
Under the assumptions of the previous proposition, in \cite{ab} the following stronger result has been proved:
\begin{equation*}
	\limsup_{h \rightarrow \infty} \frac{1}{\varepsilon_h} \int_A \int_{A'} J_{\varepsilon_h}(y - x)
	|u_h(y) - u_h(x)| \, dy \, dx  \leq C \int_\Sigma |v'(x) - v(x)| \, d\HN-1(x).
\end{equation*}
}
\end{rem}

\begin{prop} \label{P2.5}
Let sequences $\varepsilon_h \rightarrow 0^+$ and $u_h: A \rightarrow \real$ be given such
that $u_h \rightarrow u \in L^1(A)$. Given a Lipschitz function
$g: A \rightarrow \real$, then, up to subsequences, the  $\varepsilon_h$-traces of $u_h$ relative to $A$
converge to $u$ on $\{ g = t \}$ for almost every $t \in \real$.
\end{prop}

\section{Compactness and \(\Gamma\)-convergence}
This section contains the main results of the paper. In what follows we state the compactness of sequences 
$(u_h),(\rho_h)$ with equibounded energy $F_{\varepsilon_h}(u_h,\rho_h,\Omega)$ as well as the 
$\Gamma$-convergence result for $F_\varepsilon$ as $\varepsilon \rightarrow 0^+$.  We recall that, 
as it is customary in this framework (see \cite{b,dm}), we say that a family of functionals  $E_\varepsilon$
depending on a real parameter $\varepsilon > 0$ $\Gamma$-converge to a functional $E$ and we write 
$\gammalim_{\varepsilon \rightarrow 0^+} E_\varepsilon = E$ if  for any sequence of positive 
numbers $(\varepsilon_h)_h$ such that $\varepsilon_h \rightarrow 0^+$,
the $\Gamma$-limit of $E_{\varepsilon_h}$ exists and equals $E$. The same is intended for
$\gammalimsup$ and $\gammaliminf$.

\begin{thm}[Compactness] \label{compactness}
Given sequences $\varepsilon_h \rightarrow 0^+$,
$u_h \in L^1(\Omega, [-1,1])$ and $\rho_h  \in L^1(\Omega, [0,\infty))$
such that $\sup_h \|\rho_h\|_{L^1(\Omega)} < \infty$ and
$\sup_h F_{\varepsilon_h}(u_h,\rho_h,\Omega) < \infty$,
where $F_{\varepsilon_h}$  is as in \cref{defenergy}, then up to subsequences
$u_h$ converge in $L^1(\Omega)$ to some $u \in \BV(\Omega, \{-1,1\})$ and $\rho_h \mathcal{L}^N$
converge in the weak$\ast$ sense to some $\mu \in \mathcal{M}(\Omega)$.
\end{thm}
\proof
The result is a direct consequence of \cite[Theorem 3.1]{ab}
and of the weak$\ast$ compactness of $\mathcal{M}(\Omega)$.\qed\\

In the following theorem we states our $\Gamma$-convergence result. As it is customary in this setting 
(see \cite{fms, ch}) the $\Gamma$-convergence  will be understood with respect to the product topology
given by the strong $L^1$ topology of the functions and the weak$\ast$ topology of the measures. 
The proof of the result is a consequence of \cref{liminf} and \cref{limsup} that will be proven in the  next two
subsections.\\

\begin{thm}\label{thm:Gammalim}
Let $W$ be a potential as in (\ref{potential}), let $J$ be a kernel as in $\ref{kernel}$,
let $F_\varepsilon$ be as in \cref{defenergyextended} and let $F$ be defined as in (\ref{limitenergy}).
Then the following $\Gamma$-convergence result holds true.
\begin{equation}
	\Gamma\text{-}\lim\limits_{\varepsilon \rightarrow 0^+} F_{\varepsilon}(u, \mu,\Omega) 
		= F(u,\mu,\Omega).
\end{equation}
\end{thm}

\subsection{\(\Gamma\)-liminf inequality}
In the following proposition we show that we can truncate the functions $u$ and $\rho$ without increasing
the energy $F_\varepsilon(u,\rho,\Omega)$. This property of the energy functional will be needed both in the
proof of the $\gammaliminf$ and in the proof of the $\gammalimsup$ inequality.\\

\begin{prop} \label{truncation}
Let us consider $\varepsilon > 0$, $A \subset \Omega$, $u_\varepsilon \in L^1(A)$ and
$\rho_\varepsilon \in L^1(A,[0,\infty))$. Denoting $u_\varepsilon^T(x) := (u_\varepsilon(x) \wedge 1) \vee -1$
and $\rho_\varepsilon^T(x) := \rho_\varepsilon(x) \wedge \int_A J_\varepsilon(y-x) 
\frac{|u_\varepsilon^T(y) - u_\varepsilon^T(x)|}{\varepsilon} dy$
it holds
\begin{equation*}
	F_\varepsilon(u_\varepsilon,\rho_\varepsilon,A) \geq F_\varepsilon(u_\varepsilon^T, \rho_\varepsilon^T,A).
\end{equation*}
Moreover, if $u_\varepsilon$ and $\rho_\varepsilon$ are such that $\sup_\varepsilon
F_\varepsilon(u_\varepsilon,\rho_\varepsilon,A)\leq C$, then
\begin{equation*}
	\lim_ {\varepsilon\rightarrow 0}\|u_\varepsilon^T - u_\varepsilon\|_{L^1(A)} = 0
\end{equation*}
\end{prop}
\begin{proof}
The proof of the first statement is straightforward and is left to the reader.
The proof of the second statement follows the same lines of 
\cite[Lemma 1.14]{ab}. We rewrite it here for the reader's convenience:
Given $\delta > 0$, we have $W(t) > c > 0$ on $\{1+\delta < |t| < M\}$ and $W(t) > C |t|$
on $\{|t| > M\}$ for some $M > 0$ and therefore
\begin{align*}
	\|u_\varepsilon - u_\varepsilon^T\|_{L^1} &\leq \delta |\Omega| 
		+ \frac{M}{c} \int_{\{ 1+\delta < u_\varepsilon(x) < M \}} W(u_\varepsilon(x)) \, dx 
		+ \frac{1}{C}\int_{ \{ |u_\varepsilon(x) > M| \} } W(u_\varepsilon(x)) \, dx \\
	& \leq \delta |\Omega| + C \int_A W(u_\varepsilon (x)) \, dx,
\end{align*}
which implies the claim using that $\delta$ is arbitrary and $F_\varepsilon$ is bounded as 
$\varepsilon \rightarrow 0$.
\end{proof}
The next elementary Lemma will be used in the proof of the liminf inequality below.
\begin{lem} \label{aeConv}
Given a sequence of functions $(f_h)$, where $f_h \in L^1([0,1])$, that satisfies
\begin{equation*}
	\sum_{h = 0}^\infty \| f_h \|_{L^1([0,1])} < \infty,
\end{equation*}
it holds that
\begin{equation*}
	\lim_{h \rightarrow \infty} f_h(t) = 0 \quad \text {for almost every } t \in [0,1].
\end{equation*}
\end{lem}
\begin{proof}
By the monotone convergence theorem we can compute
\begin{equation*}
	\int_0^1 \sum_{h = 0}^\infty |f_h(t)| \, dt =  \sum_{h = 0}^\infty \| f_h \|_{L^1([0,1])} < \infty,
\end{equation*}
which implies that up to a set of measure zero it holds that $\sum_{h = 0}^\infty |f_h(t)| < \infty$,
hence the claim.
\end{proof}

\begin{thm}[lim{-}inf inequality]\label{liminf}
Given $u \in \BV(\Omega,\{-1,1\})$, $\rho \in \mathcal{M}(\Omega)$ and given
$\varepsilon \rightarrow 0^+$, $u_\e \in L^1(\Omega)$ and
$\rho_\e \in L^1(\Omega,[0,\infty))$ such that $u_\e \rightarrow u$ in $L^1(\Omega)$ and 
$\rho_\e \weak* \mu$ in $\mathcal M(\Omega)$, it holds that
\begin{equation*}
	\liminf _{\e \rightarrow 0} F_{\varepsilon}(u_\e,\rho_\e,\Omega) \geq
	\int_{S_u} \sigma \left( \nu_{u}, \frac{d \mu}{d \HN-1 \llcorner S_u} \right) \, d\HN-1,
\end{equation*}
where $F_\varepsilon$ is the functional given by \cref{defenergy}.
\end{thm}
\begin{proof}
Using \cref{truncation} and the monotonicity of $\sigma(\nu,\cdot)$ observed in Remark
\ref{rm-sigma-decr}, it is enough to show
the claim for functions $u_\e$ and $\rho_\e$ such that
\begin{equation} \label{temp4}
	u_\e  : \Omega \rightarrow [-1,1] \quad  \text {and} \quad
	\rho_\e(x) \leq \int_\Omega J_{\varepsilon}(y-x) \frac{|u_\e(y) - u_\e(x)|}{\varepsilon} \,dy.
\end{equation}
Extracting  subsequences, we may assume that
\begin{equation}\label{liminf=lim}
	\liminf _{\e \rightarrow 0} F_{\varepsilon}(u_\e,\rho_\e,\Omega) 
	= \lim_{\e \rightarrow 0} F_{\varepsilon}(u_\e,\rho_\e,\Omega)
\end{equation}
and moreover we may assume that this limit is finite. 
We prove the liminf inequality by blow-up. We start by setting
\begin{equation*}
	g_\varepsilon(x) := \frac{1}{\varepsilon} W(u_\varepsilon(x)) + \frac{1}{\varepsilon} \hspace{-.1cm}
	\int_\Omega\hspace{-.1cm} J_\varepsilon(y - x) (u_\varepsilon(y) - u_\varepsilon(x))^2 dy +  \varepsilon
	\left( \int_\Omega \hspace{-.1cm}J_\varepsilon(y - x) \frac{|u_\varepsilon(y) - u_\varepsilon(x)|}
	{\varepsilon} \,dy - \rho_\varepsilon(x) \right)^2
\end{equation*}
and note that $\int_\Omega g_\e(x)\,dx=F_\e(u_\e,\rho_\e,\Omega)$. Hence, by \eqref{liminf=lim}, passing to
a subsequence if necessary, we can assume that  $g_\varepsilon \mathcal{L}^N \weak* \lambda$ for a
measure $\lambda \in \mathcal{M}(\Omega)$. In order to prove the statement of the theorem, it is then
enough to show that
\begin{equation*}
	\frac{d \lambda}{d \HN-1 \llcorner S_u}(x_0) \geq \sigma \left(\nu_{u}(x_0),\frac{d \mu}
	{d \HN-1 \llcorner S_u}(x_0)\right)
	\text { for } \HN-1 \text{ a.e. } x_0 \in S_u.
\end{equation*}
In what follows we refer to \cite{afp} for all the results about the blow-up properties of functions of bounded
variations. Let us denote by $\nu_{u}(x_0)$ the generalized normal at $x_0 \in S_u$ which exists at 
$\HN-1$ a.e. point. We denote by $Q^l_{\nu_u(x_0)}$ the cube obtained by rotating $[-l,l]^N$ around the
origin in such a way that one axis is in direction $\nu_{u}(x_0)$ and assume that $l>0$ is chosen such that
$x_0+Q^l_{\nu_u(x_0)}\subset\Omega$. Denoting by $u_b$ the jump function
\begin{equation*}
	u_b(x) := \begin{cases}
	u^+(x_0), &(x, \nu_{u}(x_0)) > 0\\
	u^-(x_0), &(x, \nu_{u}(x_0)) < 0,
	\end{cases} 
\end{equation*}
it holds that 
\begin{equation}\label{blow-up-u}
	 \lim_{\delta \rightarrow 0^+} \int_{x_0+Q^l_{\nu_u(x_0)}} |u(\delta x) - u_b(x)| \, dx = 0.
\end{equation}

In what follows we fix any sequence $\varepsilon_h \rightarrow 0^+$.
We claim that there exists a vanishing sequence of positive real numbers $(\delta_m)$ and a set 
$S \subset [0,1]$ (depending on this sequence) with $|S|=1$ such that for all $t \in S$
\begin{equation} \label{temp1}
	\lim_{m \rightarrow \infty} \delta_m^{1-N} \int_{\partial Q^{l\delta_m t}_{\nu_u(x_0)}}
	|u(x)-u_b(x)| \, d\HN-1(x) =
	 \lim_{m \rightarrow \infty} \int_{\partial Q^{l t}_{\nu_u(x_0)}} |u(\delta_m x) - u_b(x)| \, d\HN-1(x) = 0
\end{equation}
and such that the $\varepsilon_h$-traces
of $u_h(\delta_m \cdot)$ relative to $Q^{l}_{\nu_u(x_0)}$ converge to
$u(\delta_m \cdot) |_{\partial Q^{l}_{\nu_u(x_0)}}$ as $h \rightarrow \infty$.

We start by observing that, thanks to \eqref{blow-up-u}
there exist $\delta_m \rightarrow 0^+$ such that
\begin{equation*} 
	(2l)^{N-1} \int_0^1 \int_{x_0+\partial Q^{lt}_{\nu_u(x_0)}} |u(\delta_m x) - u_b(x)| \, d\HN-1(x) dt
	= \int_{x_0+Q^l_{\nu_u(x_0)}}  |u(\delta_m x) - u_b(x)| \, dx < 2^{-m}.
\end{equation*}

Moreover, from \cref{P2.5} we know that, up to choosing a subsequence of $\varepsilon_h$, 
for all but at most countably many $\delta$, 
the $\varepsilon_h$-traces of $u_h(\delta \cdot)$ relative to $x_0+Q^{l}_{\nu_u(x_0)}$ converge to 
$u(\delta \cdot) |_{x_0+\partial Q^{l}_{\nu_u(x_0)}}$ as 
$h \rightarrow \infty$ and therefore the $\delta_m$ can be chosen to have this property.
The claim follows by an application of \cref{aeConv} to the functions
$f_m(t) := \int_{{x_0+\partial Q^{lt}_{\nu_u(x_0)}}} |u(\delta_m x) - u_b(x)| \, d\HN-1(x)$.
For $\HN-1$ a.e. $x_0$ and for all but at most countably many $t  \in [0,1]$  it holds that
$\mu(x_0+ \partial Q^{l \delta_m t}_{\nu_u(x_0)}) = \lambda(x_0+ \partial Q^{l \delta_m t}_{\nu_u(x_0)}) 
= 0$.
Therefore we can find a $t \in S$ such that
\begin{equation*}
	\mu(x_0+ \partial Q^{l \delta_m t}_{\nu_u(x_0)}) 
	= \lambda(x_0+ \partial Q^{l \delta_m t}_{\nu_u(x_0)}) = 0 \quad \text{for all } m.
\end{equation*}
We fix such a $t$ and write $Q_{\delta_m} := x_0 + Q^{l \delta_m t}_{\nu_u(x_0)}$ and we set 
$\omega:=(2lt)^{N-1}$ for the $\HN-1$-measure of a midplane of $ x_0 + Q^{l t}_{\nu_u(x_0)}$,
i.e., the measure of any face of $ x_0 + Q^{l t}_{\nu_u(x_0)}$.
It holds
\begin{align*}
&\mu(\partial Q_{\delta_m}) = 0 \Rightarrow \lim_{h \rightarrow \infty} 
		\int_{Q_{\delta_m}} \rho_h \, dx = \mu(Q_{\delta_m}), \\
&\lambda(\partial Q_{\delta_m}) = 0 \Rightarrow \lim_{h \rightarrow \infty} 
		F_{\varepsilon_h}(u_h,\rho_h,Q_{\delta_m}) = \lambda(Q_{\delta_m}), \\
&\lim_{m \rightarrow \infty} \frac{\HN-1 \llcorner S_u (Q_{\delta_m})}{\omega \delta_m^{N-1}} = 1, \\
&\lim_{m \rightarrow \infty} \frac{\mu(Q_{\delta_m})}{\HN-1 \llcorner S_u (Q_{\delta_m})}
	= \frac{d \mu}{d \HN-1 \llcorner S_u}(x_0), \\
&\lim_{m \rightarrow \infty} \frac{\lambda(Q_{\delta_m})}{\HN-1 \llcorner S_u (Q_{\delta_m})}
	= \frac{d \lambda}{d \HN-1 \llcorner S_u}(x_0).
\end{align*}
In particular, given $s > 0$, we have
\begin{align} \label{temp5}
	&\int_{Q_{\delta_m}} \rho_h \,dx \leq \omega \delta_m^{N-1} \frac{d \mu}{d \HN-1
		\llcorner S_u}(x_0) +s \nonumber \\
	&\frac{d \lambda}{d \HN-1 \llcorner S_u}(x_0) + s
		 \geq \lim_{h \rightarrow \infty} F_{\varepsilon_h}(u_h,\rho_h, Q_{\delta_m})
		\omega^{-1} \delta_m^{1-N}
\end{align}
for all $m > m_0$ (depending on $s$).
We now fix a $\delta \in \{ \delta_m \, : \, m \in \mathbb{N}  \}$, denote by $C_\delta$ the intersection
of $Q_\delta$ with the hyperplane through $x_0$ and orthogonal to $ \nu_{u}(x_0)$ and
define $C_\delta$-periodic functions $\tilde u_h: \real^N \rightarrow [-1,1]$,
$\tilde u: \real^N \rightarrow [-1,1]$ and $\tilde \rho_h: \real^N \rightarrow \real$ which in
$T_\delta:=T_{C_\delta}$ are given by
\begin{equation*} 
	\tilde u_h(x)=
	\begin{cases}
		u_h(x), &x \in Q_\delta \\
		u_b(x), &x \in T_{\delta} \setminus Q_\delta,
	\end{cases}\,\,
	 \tilde u(x)=
		\begin{cases}
		u(x), &x \in Q_\delta \\
		u_b(x), &x \in T_{\delta} \setminus Q_\delta,
		\end{cases}\,\,
		\tilde \rho_h(x)=
		\begin{cases}
		\rho_h(x), &x \in Q_\delta \\
		0, &x \in T_{\delta} \setminus Q_\delta.
		\end{cases}
\end{equation*}
Thanks to the $C_\delta$-periodicity of  $\tilde u_h$, as $h \rightarrow \infty$
the $\varepsilon_h$-traces of $\tilde u_h$ relative to $T_\delta$ (resp. relative to 
$\real^N \setminus T_\delta$) converge to $v: \partial T_\delta \rightarrow \real$ defined as

\begin{equation*}
	v(x)=\begin{cases}
 		u_{|\partial T_\delta}(x), & x \in Q_\delta\\
		{u_b}_{|\partial T_\delta}(x), & x \in T_\delta \setminus Q_\delta
\end{cases}
\end{equation*}
(resp. to	$\bar v: \partial T_\delta \rightarrow \real$, $\bar v(x)=v (\bar x)$
where $\bar x$ is the point obtained reflecting $x$ with respect to the hyperplane
through $x_0$ that is parallel to the face on which $x$ lies).
The $\varepsilon_h$-traces of $\tilde u_h$ relative to $Q_\delta$ resp. relative to 
$T_\delta \setminus Q_\delta$ converge to $u_{|\partial Q_\delta}$ resp. to
${u_b}_{|\partial T_\delta \setminus Q_\delta}$.
By the very definition of the energy functionals given in Definition \ref{defenergy} we have that
\begin{equation*}
	0 \leq \mathcal{F}_{\varepsilon_h}(\tilde u_h, \tilde \rho_h, T_\delta) - 
		F_{\varepsilon_h}(\tilde u_h, \tilde \rho_h, Q_\delta)=I_1+I_2+I_3,
\end{equation*}
where we have set
\begin{equation*}
I_1= \varepsilon_h \int_{T_\delta} \int_{\real^N}\hspace{-.2cm}J_{\varepsilon_h}(y-x) 
		\bigg( \frac{|\tilde u_h(y) - \tilde u_h(x)|}{\varepsilon_h} \bigg)^2 dy \,dx \\
	-\varepsilon_h \int_{Q_\delta} \int_{Q_\delta}\hspace{-.1cm}J_{\varepsilon_h}(y-x) 
		\bigg( \frac{|\tilde u_h(y) - \tilde u_h(x)|}{\varepsilon_h} \bigg)^2 dy \,dx , 
\end{equation*}

\begin{equation*}
I_2= \varepsilon_h \int_{T_\delta \setminus Q_\delta}\left( \int_{\real^N} J_{\varepsilon_h} (y-x)
		\frac{|\tilde u_h(y) - \tilde u_h(x)|}{\varepsilon_h}\,dy - \tilde \rho_h(x) \right)^2 dx,
\end{equation*}
\begin{eqnarray*}
I_3= &\varepsilon_h \displaystyle\int_{Q_\delta} \bigg( \int_{\real^N} J_{\varepsilon_h} (y-x)
		\frac{|\tilde u_h(y) - \tilde u_h(x)|}{\varepsilon_h}\,dy - \tilde \rho_h(x) \bigg)^2 dx \\
	&-\varepsilon_h \displaystyle\int_{Q_\delta} \bigg( \int_{Q_\delta} J_{\varepsilon_h} (y-x)
		\frac{|\tilde u_h(y) - \tilde u_h(x)|}{\varepsilon_h}\,dy - \tilde \rho_h(x) \bigg)^2 dx.
\end{eqnarray*}
The term $I_1$ can be estimated using $|\tilde u_h| \leq 1$ to get that
\begin{align*}
I_1 = \,& \varepsilon_h \int_{T_\delta} \int_{\real^N \setminus T_\delta}\hspace{-.5cm}J_{\varepsilon_h}(y-x) 
		\left( \frac{|\tilde u_h(y) - \tilde u_h(x)|}{\varepsilon_h} \right)^2 dy \,dx 
	+  \varepsilon_h \int_{T_\delta \setminus Q_\delta} 
		\int_{T_\delta \setminus Q_\delta}\hspace{-.5cm}J_{\varepsilon_h}(y-x) 
		\left( \frac{|\tilde u_h(y) - \tilde u_h(x)|}{\varepsilon_h} \right)^2 dy \,dx \\
	&+ 2 \varepsilon_h \int_{T_\delta \setminus Q_\delta} \int_{Q_\delta}J_{\varepsilon_h}(y-x) 
		\left( \frac{|\tilde u_h(y) - \tilde u_h(x)|}{\varepsilon_h} \right)^2 dy \,dx \\
	\leq \,& C  \int_{T_\delta} \int_{\real^N \setminus T_\delta}J_{\varepsilon_h}(y-x) 
		 \frac{|\tilde u_h(y) - \tilde u_h(x)|}{\varepsilon_h} \,dy \,dx 
	+  C\int_{T_\delta \setminus Q_\delta} 
		\int_{T_\delta \setminus Q_\delta}J_{\varepsilon_h}(y-x) 
		 \frac{|\tilde u_h(y) - \tilde u_h(x)|}{\varepsilon_h} \,dy \,dx \\
	&+ C\int_{T_\delta \setminus Q_\delta} \int_{Q_\delta}J_{\varepsilon_h}(y-x) 
		 \frac{|\tilde u_h(y) - \tilde u_h(x)|}{\varepsilon_h} \,dy \,dx,
\end{align*}
To estimate the other two terms $I_2$ and $I_3$ one uses $|\tilde u_h| \leq 1$, \eqref{temp4}
and Jensen's inequality to find out that
\begin{align*}
	I_2 &= \varepsilon_h \int_{T_\delta \setminus Q_\delta} \left( \int_{\real^N} J_{\varepsilon_h} (y-x)
		\frac{|\tilde u_h(y) - \tilde u_h(x)|}{\varepsilon_h}\,dy - \tilde \rho_h(x) \right)^2 dx \\
	&\leq  \varepsilon_h  \int_{T_\delta \setminus Q_\delta} \left( \int_{\real^N} J_{\varepsilon_h} (y-x)
		\frac{|\tilde u_h(y) - \tilde u_h(x)|}{\varepsilon_h}\,dy \right)^2 dx \\
	&\leq \|J\|_{L^1(\real^d)} \varepsilon_h \int_{T_\delta \setminus Q_\delta} 
		 \int_{\real^N} J_{\varepsilon_h} (y-x)
		\left( \frac{|\tilde u_h(y) - \tilde u_h(x)|}{\varepsilon_h} \right)^2 \,dy \,dx \leq \|J\|_{L^1(\real^N)} I_1.
\end{align*}
and that
\begin{align*}
	I_3 &= \varepsilon_h \int_{Q_\delta} \left( \int_{\real^N \setminus Q_\delta} J_{\varepsilon_h} (y-x)
		\frac{|\tilde u_h(y) - \tilde u_h(x)|}{\varepsilon_h}\,dy  \right)^2 dx \\
	&\quad +2\varepsilon_h \int_{Q_\delta}  \int_{\real^N \setminus Q_\delta} J_{\varepsilon_h} (y-x)
		\frac{|\tilde u_h(y) - \tilde u_h(x)|}{\varepsilon_h}\,dy
		\left( \int_{Q_\delta} J_{\varepsilon_h} (y-x)
		\frac{|\tilde u_h(y) - \tilde u_h(x)|}{\varepsilon_h}\,dy - \tilde \rho_h(x) \right)dx \\
	&\leq \|J\|_{L^1(\real^N)} \varepsilon_h \int_{Q_\delta}  \int_{\real^N \setminus Q_\delta}
		 J_{\varepsilon_h} (y-x)
		\left( \frac{|\tilde u_h(y) - \tilde u_h(x)|}{\varepsilon_h} \right)^2dy  \,dx \\
	&\quad +2\varepsilon_h \int_{Q_\delta}  \int_{\real^N \setminus Q_\delta} J_{\varepsilon_h} (y-x)
		\frac{|\tilde u_h(y) - \tilde u_h(x)|}{\varepsilon_h}\,dy
		 \int_{Q_\delta} J_{\varepsilon_h} (y-x)
		\frac{|\tilde u_h(y) - \tilde u_h(x)|}{\varepsilon_h}\,dy\,dx \\
	&\leq C \int_{Q_\delta}  \int_{\real^N \setminus Q_\delta} J_{\varepsilon_h} (y-x)
		\frac{|\tilde u_h(y) - \tilde u_h(x)|}{\varepsilon_h}\,dy\,dx \\
	&\leq C \int_{T_\delta}  \int_{\real^N} 
		J_{\varepsilon_h} (y-x) \frac{|\tilde u_h(y) - \tilde u_h(x)|}{\varepsilon_h}\,dy\,dx
		-C \int_{Q_\delta}  \int_{Q_\delta} J_{\varepsilon_h} (y-x)
		\frac{|\tilde u_h(y) - \tilde u_h(x)|}{\varepsilon_h}\,dy\,dx  \\
	&\leq C \int_{T_\delta} \int_{\real^N \setminus T_\delta}J_{\varepsilon_h}(y-x) 
		 \frac{|\tilde u_h(y) - \tilde u_h(x)|}{\varepsilon_h} \,dy \,dx
		+  C\int_{T_\delta \setminus Q_\delta} 
		\int_{T_\delta \setminus Q_\delta}J_{\varepsilon_h}(y-x) 
		 \frac{|\tilde u_h(y) - \tilde u_h(x)|}{\varepsilon_h} \,dy \,dx \\
	&\quad + C\int_{T_\delta \setminus Q_\delta} \int_{Q_\delta}J_{\varepsilon_h}(y-x) 
		 \frac{|\tilde u_h(y) - \tilde u_h(x)|}{\varepsilon_h} \,dy \,dx 
\end{align*}
By an application of \cref{L2.7} (see Remark \ref{rm-linear-traces}) we also have that 
\begin{align}\label{temp-reminder}
	\lim_{h \rightarrow \infty}& \int_{T_\delta} \int_{\real^N \setminus T_\delta}J_{\varepsilon_h}(y-x) 
		 \frac{|\tilde u_h(y) - \tilde u_h(x)|}{\varepsilon_h} \,dy \,dx 
		\leq C \int_{\partial T_\delta} |v - \bar v| \,d\HN-1, \\ \nonumber
	\lim_{h \rightarrow \infty}& \int_{T_\delta \setminus Q_\delta} \int_{Q_\delta}J_{\varepsilon_h}(y-x) 
		 \frac{|\tilde u_h(y) - \tilde u_h(x)|}{\varepsilon_h} \,dy \,dx 
		\leq C \int_{T_\delta \cap \partial Q_\delta} |u_{|\partial Q_\delta} - u_b| \,d\HN-1, 
	\\ \nonumber
	\lim_{h \rightarrow \infty}& \int_{T_\delta \setminus Q_\delta}
		\int_{T_\delta \setminus Q_\delta}J_{\varepsilon_h}(y-x) 
		\frac{|\tilde u_h(y) - \tilde u_h(x)|}{\varepsilon_h} \,dy \,dx \\ \nonumber
	&= \lim_{h \rightarrow \infty}2 \int_{(T_\delta \setminus Q_\delta) \cap \{ (x,\nu_{u}(x_0)) > 0 \}}
		\int_{(T_\delta \setminus Q_\delta) \cap \{ (x,\nu_{u}(x_0)) < 0 \}}
		J_{\varepsilon_h}(y-x) \frac{|\tilde u_h(y) - \tilde u_h(x)|}{\varepsilon_h} \,dy \,dx = 0,
\end{align}
where in the last equation we used that the sets
$(T_\delta \setminus Q_\delta) \cap \{ (x,\nu_{u}(x_0)) > 0 \}$ and
$(T_\delta \setminus Q_\delta) \cap \{ (x,\nu_{u}(x_0)) < 0 \}$ have positive distance,
which implies that the limit of the double integral is zero
(this can be shown by extending $\tilde u_h$ by $0$ on $Q_\delta$).
Defining $\bar \rho_h:\real^N \to \real$ as the $C_\delta$-periodic function such that
\begin{equation*}
	\bar \rho_h(x)=\begin{cases}
		\rho_h(x) - \displaystyle\frac{s}{|Q_\delta|}, &x \in Q_\delta \\
		0, &x \in T_{\delta} \setminus Q_\delta,
		\end{cases}
\end{equation*}
we have that\\
\begin{align}\label{temp-reminder2}
	| \mathcal{F}_{\varepsilon_h}& (\tilde u_h,\bar\rho_h,Q_\delta) - 
		 \mathcal{F}_{\varepsilon_h} (\tilde u_h,\tilde \rho_h,Q_\delta)| \\ \nonumber
	&\leq \varepsilon_h \left|2 \int_{Q_\delta} \int_{\real^N} J_{\varepsilon_h}(y -x) 
		\frac{|\tilde u_h(y) - \tilde u_h(x)|}{\varepsilon_h}\,dy \frac{s}{|Q_\delta|} \,dx \right|
		+ \varepsilon_h \left |\int_{Q_\delta} \left(\frac{s}{|Q_\delta|} \right)^2 dx \right| \\ \nonumber
	&\leq \sqrt{\varepsilon_h} C(\delta) + \varepsilon_h \frac{Cs^2}{|Q_\delta|} 
		\rightarrow_{h \rightarrow \infty}0,
\end{align}
where we used the boundedness of $(\mathcal{F}_{\varepsilon_h}(\tilde u_h, 0,Q_\delta))_h$ which is a
consequence of the second inequality in \eqref{temp4} and the boundedness of 
$(\mathcal{F}_{\varepsilon_h}(\tilde u_h, \rho_h,Q_\delta))_h$.
By the very definition of $\bar \rho_h$, Definition \ref{temp7} and (\ref{temp5}), it holds true that
\begin{equation*}
	(\tilde u_h,\bar \rho_h) \in \mathcal{A}\left(\nu_{u}(x_0), \frac{d\mu}{d\HN-1 \llcorner S_u}(x_0)\right),
	\quad \mathcal{F}_{\varepsilon_h} (\tilde u_h,\bar\rho_h,Q_\delta) \omega^{-1}\delta^{1-N}
	\geq \sigma \left(\nu_{u}(x_0),\frac{d\mu}{d\HN-1 \llcorner S_u}(x_0)\right)
\end{equation*}
Therefore, collecting the above results,
\begin{align*}
	\frac{d \lambda}{d \HN-1 \llcorner S_u} (x_0) + s & \geq \lim_{h \rightarrow \infty} F_{\varepsilon_h}
		(u_h,\rho_h,Q_{\delta_m}) \omega^{-1} \delta_m^{1-N} \\
	&=  \lim_{h \rightarrow \infty} F_{\varepsilon_h} 
		(\tilde u_h,\tilde \rho_h,Q_{\delta_m}) \omega^{-1} \delta_m^{1-N} \\
	&\geq \liminf_{h \rightarrow \infty} \mathcal{F}_{\varepsilon_h} 
		(\tilde u_h,\bar \rho_h,Q_{\delta_m}) \omega^{-1} \delta_m^{1-N} \\
	&\quad - \limsup_{h \rightarrow \infty} |\mathcal{F}_{\varepsilon_h} 
		(\tilde u_h,\tilde \rho_h,Q_{\delta_m}) \omega^{-1} \delta_m^{1-N} - \mathcal{F}_{\varepsilon_h} 
		(\tilde u_h,\bar \rho_h,Q_{\delta_m}) \omega^{-1} \delta_m^{1-N}|\\
	&\quad - \limsup_{h \rightarrow \infty} |F_{\varepsilon_h} 
		(\tilde u_h,\tilde \rho_h,Q_{\delta_m}) \omega^{-1} \delta_m^{1-N} - \mathcal{F}_{\varepsilon_h} 
		(\tilde u_h,\tilde \rho_h,Q_{\delta_m}) \omega^{-1} \delta_m^{1-N}|\\
	&\geq \sigma \left(\nu_{u}(x_0),\frac{d \mu}{d \HN-1 \llcorner S_u}(x_0)\right) -
			C \delta_m^{1-N} \int_{\partial T_{\delta_m}} |v - \bar v| d\HN-1 \\
	&\quad -C \delta_m^{1-N} \int_{T_{\delta_m} \cap\,\partial Q_{\delta_m}} 
		|u_{|\partial Q_{\delta_m}} - u_b| d\HN-1.
\end{align*}
Here in the last estimate we have used \eqref{temp-reminder2} and \eqref{temp-reminder}. Passing to 
the limit as $m \rightarrow \infty$ in the estimate above by \eqref{temp1} we eventually get that for 
$\HN-1$ a.e. $x_0 \in S_u$
\begin{equation*}
	\frac{d \lambda}{d \HN-1 \llcorner S_u}(x_0)+s
	\geq \sigma \left( \nu_{u}(x_0),\frac{d \mu}{d \HN-1 \llcorner S_u}(x_0) \right),
\end{equation*}
which proves the claim thanks to the arbitrariness of $s$.
\end{proof}

\subsection{\(\Gamma\)-limsup inequality}
In this section we complete the proof of the $\Gamma$-convergence result by proving the limsup
inequality for the energy functionals $F_\e$ in \cref{defenergy}. The proof consists of several steps. 
At each step we assume $u$ and $\mu$ to be of increasing generality and provide for them a recovery
sequence. We start with the case of polyhedral functions (see \cref{D5.1}),
the proof of which draws some inspiration from 
\cite[Theorem 5.2]{ab} and is given
in \cref{limsupPolyhedral}. We further generalize it in \cref{temp8}. Therein, we consider functions
of bounded variation $u$ whose jump set need not be polyhedral, while we further keep restrictions on 
the limit measure $\mu$. Finally, in \cref{limsup}, we show the statement in  full generality.

\begin{prop} \label{limsupPolyhedral}
Let $u \in \BV(\Omega, \{-1,1\})$ be a polyhedral function corresponding to a polyhedral set $P$
as in \cref{D5.1}. For $n\in\N$ and $i\in\{1,2,\dots,n\}$ let $x_i \in \Omega \setminus S_u$, $g$
be a piecewise constant function on polyhedral subsets of the faces 
of $S_u$ and let $\mu = \sum_{i = 1}^n \beta_i \delta_{x_i} + g
\HN-1 \llcorner S_u$. Then there exist functions
$u_\varepsilon \in {L^1}(\Omega) $ and  $\rho_\varepsilon \in {L^1}(\Omega,[0,\infty))$, 
such that  $u_\varepsilon \rightarrow u$ in $L^1(\Omega)$
and $\rho_\varepsilon \mathcal{L}^N \weak* \mu$  as $\varepsilon \rightarrow 0^+$ which satisfy
\begin{equation} \label{temp2}
	\limsup _{\varepsilon \rightarrow 0^+} F_{\varepsilon}(u_\varepsilon, \rho_\varepsilon,\Omega) 
	\leq \int_{S_u} \sigma \left( \nu_{u}, \frac{d \mu}{d \HN-1 \llcorner S_u} \right) \,d\HN-1.
\end{equation}
\end{prop}
\begin{proof}
Let us assume without loss of generality (we may reduce to this case 
by the argument of step 3 below) that 
\begin{equation*}
\mu = \beta \delta_{x_1} + \gamma \HN-1 \llcorner \Sigma_g,
\end{equation*}
where $\Sigma_g$ is a polyhedral subset (the subscript $g$ points out the dependence of this set 
on the function $g$) of a face of $P$.

\begin{figure}[hbt!]
\begin{tikzpicture}
	\begin{scope}
		\clip(0,0) ellipse [x radius = 3, y radius = 2.2];
		\fill [fill = black!12] (-3,-2.5) -- (-0.5,-2.5) -- (-0.5, 2.5) -- (-3,2.5) -- cycle;
		\fill [fill = black!12] (0.5,0) -- (0.5,-2.5) -- (3,-2.5) --(3,2.5) -- cycle;
		\fill [fill = black!20] (0.5,0) -- (0.5,2.5) --(3,2.5) -- cycle;
	\end{scope}
	\draw(0,0) ellipse [x radius = 3, y radius = 2.2];
	\draw (-3,0) -- (0.5,0) -- (2.7,-1);
	\draw[line width = 1.2 pt] (-3,0) -- (-0.5,0);
	\draw (-0.5,1) node {$\Omega$};
	\draw (1.5,-0.2) node {$\partial P$};
	\draw (-2,0.2) node {$\Sigma_g$};
	\draw (1,2.7) node {$A$ as in step 1};
	\draw [dashed] (1,2.5) -- (1,1);
	\draw (0,-3) node {$A$ as in step 2};
	\draw [dashed] (-0.5,-2.7) -- (-2, -1);
	\draw [dashed] (0,-2.7) -- (0, -1);
	\draw [dashed] (0.5,-2.7) -- (1.7, -1);
\end{tikzpicture}
\caption{Sketch of the situation of Proposition \ref{limsupPolyhedral}}\label{fig:step1and2}
\end{figure}\begin{itemize}

\item[Step 1.]
	We first consider the case of a polyhedral set $A$ in $\Omega$ that does not intersect the boundary of
	$P$ (see Figure \ref{fig:step1and2})
	and define sequences $u_\varepsilon \rightarrow u$ in $L^1(A)$ and
	$\rho_\varepsilon \mathcal{L}^N \weak* \mu$ in $A$ that satisfy the (local) energy estimate
	\begin{equation}\label{limsup_local}
		\limsup _{\varepsilon \rightarrow 0^+} F_{\varepsilon}(u_\varepsilon, \rho_\varepsilon,A) 
		\leq \int_{S_u |_A} \sigma \left( \nu_{u}, \frac{d \mu}{d \HN-1 \llcorner S_u} \right) \,d\HN-1.
	\end{equation}
	In this case we notice that $u_{|A}$ is constant and that the r.h.s. of \eqref{limsup_local} is zero. 
	Hence it suffices to define
	$u_\varepsilon := u_{|A}$ and 
	$\rho_\varepsilon := \beta\,\omega_N^{-1} \varepsilon^{-1/2}
	\chi_{B(x_1,{\varepsilon^{1/(2N)}})}$, if $x_1 \in A$
	 and $\rho_\varepsilon := 0$ 
	if instead $x_1 \not \in A$ to have
	\begin{equation*}
		u_\varepsilon \rightarrow u, \quad \rho_\varepsilon \mathcal{L}^N \weak* \mu.
	\end{equation*}
	The limsup inequality \ref {limsup_local} holds true since
	\begin{equation*}
		 F_{\varepsilon} (u_\varepsilon, \rho_\varepsilon, A) \leq\varepsilon 
		\int_{B(x_1,{\varepsilon^{1/(2N)}})}
		\frac{\beta^2}{\omega_N^2 \varepsilon} dx \leq C \varepsilon^{1/2}.
	\end{equation*}

\item[Step 2.]
	We now take a polyhedral set $A$ in $\Omega$ such that
	it intersects exactly one face of $P$ and the orthogonal projection of $A$ on the hyperplane 
	which contains this face of $P$ equals the intersection $\partial P \cap A =: \Sigma$.
	In what follows, we define sequences $u_\varepsilon \rightarrow u$ in $L^1(A)$ and
	$\rho_\varepsilon \mathcal{L}^N \weak* \mu$ in $A$ such that the (local) energy estimate
	 \eqref{limsup_local}
	holds true.
	Considering step 3 below, we can assume without loss of generality that $\Sigma \subset \Sigma_g$
	or $\Sigma \cap \Sigma_g = \emptyset$. We now assume that $\Sigma \subset \Sigma_g$, the other
	case can be treated similarly. In order to simplify the notation we
	also assume that $\Sigma \subset \{ x_N = 0 \}$.
	Given any $\eta > 0$, we find $(v,\tilde \rho) \in \mathcal{A}(e_N, \gamma)$ such that for a
	$(N-1)$-dimensional cube
	$C \subset \{ x_N = 0 \}$ it holds
	$\HN-1(C)^{-1} \mathcal{F} (v,\tilde \rho, T_C) \leq \sigma(e_N,\gamma) + \eta$.
	We define
	\begin{equation*}
		u_\varepsilon(x) := v\left(\frac{x}{\varepsilon}\right), \quad \tilde \rho_\varepsilon (x) := 
		\frac{1}{\varepsilon} \tilde \rho \left( \frac{x}{\varepsilon} \right)
	\end{equation*}
	and
	\begin{align*}
		r_\varepsilon(x)& := \frac{\beta}{\omega_N \varepsilon^{1/2}}
			\chi_{B_{\varepsilon^{1/(2N)}}(x_1)} (x)
			+ \frac{\gamma - (\HN-1(C))^{-1} \int_{T_C} \tilde \rho dx}{\varepsilon^{1/2}}
			\chi_{\{ x: 0 < x_N < \varepsilon^{1/2} \}}(x) \\
		\rho_\varepsilon(x) &:= \tilde \rho_\varepsilon (x) + r_\varepsilon(x).
	\end{align*}
	Remark that by the Riemann-Lebesgue Lemma it holds
	\begin{equation*}
		\tilde \rho_\varepsilon \mathcal{L}^N \llcorner A \weak* (\HN-1(C))^{-1} \int_{T_C}\tilde \rho \,dx 
		\HN-1 \llcorner \Sigma,
	\end{equation*}
	hence the convergence $\rho_\varepsilon \mathcal{L}^N \llcorner A \weak* \mu \llcorner A$ holds true.
	We note that, by the very definition of $\mathcal{A}(e_N, \gamma)$, one has that $r_\varepsilon \geq 0$. 
	Since $|u_\e|\leq 1$, using \cref{truncation}, we may assume that
	$\int_A J_{\varepsilon} (y-x) \frac{|u_\varepsilon(y) -u_\varepsilon(x)|}{\varepsilon} dy
	 - \tilde \rho_\varepsilon(x) \geq 0$ and have that
	\begin{align*}
		0\leq F_{\varepsilon}&(u_\varepsilon, \rho_\varepsilon, A) -  F_{\varepsilon}
			(u_\varepsilon, \tilde \rho_\varepsilon, A) \\
		& = 2 \varepsilon \int_A \left( \int_A J_{\varepsilon} (y-x)
			 \frac{|u_\varepsilon(y) -u_\varepsilon(x)|}{\varepsilon}
			\,dy - \tilde \rho_\varepsilon(x) \right) (-r_\varepsilon(x)) \,dx + \varepsilon 
			\int_A r_\varepsilon(x)^2 \,dx \\
		& \leq \varepsilon \int_A r_\varepsilon(x)^2 \,dx \leq C \varepsilon^{1/2}.
	\end{align*}
	The estimate above implies that
	\begin{equation} \label{temp3}
		\limsup_{\varepsilon \rightarrow 0^+} F_{\varepsilon}(u_\varepsilon, \rho_\varepsilon, A)
		\leq \limsup_{\varepsilon \rightarrow 0^+} F_{\varepsilon}(u_\varepsilon, \tilde \rho_\varepsilon, A).
	\end{equation}
	We notice that
	\begin{equation*}
		\mathcal{F}_{\varepsilon}(u_\varepsilon ,\tilde \rho_\varepsilon,T_{\varepsilon C}) =
		\varepsilon^{N-1} \mathcal{F}(u,\tilde \rho, T_C).
	\end{equation*}
	Given a covering (depending on $\varepsilon$) of $\Sigma$ made of $M_\varepsilon$
	translated copies of $\varepsilon C$, this equality implies
	\begin{equation*}
		F_{\varepsilon}(u_\varepsilon ,\tilde \rho_\varepsilon, A) \leq M_\varepsilon \varepsilon^{N-1}
		\mathcal{F}(u,\tilde \rho, T_C).
	\end{equation*}
	Further choosing the covering such that
	$M_\varepsilon \HN-1(C) \varepsilon^{N-1} \rightarrow \HN-1(\Sigma)$ as $\varepsilon \rightarrow 0^+$,
	we have shown that
	\begin{equation*}
		\limsup_{\varepsilon \rightarrow 0^+}F_{\varepsilon}(u_\varepsilon,\tilde \rho_\varepsilon, A) 
		\leq \HN-1(\Sigma)
		(\HN-1(C))^{-1} \mathcal{F} (v,\tilde \rho, T_C) \leq \HN-1(\Sigma) \sigma(e_N,\gamma) 
		+\HN-1(\Sigma) \eta.
	\end{equation*}
	The latter estimate together with (\ref{temp3}) eventually gives 
	\begin{equation*}
		\limsup_{\varepsilon \rightarrow 0^+} F_{\varepsilon}(u_\varepsilon, \rho_\varepsilon, A) \leq
			\HN-1(\Sigma) \sigma(e_N,\gamma)  +C \eta,
	\end{equation*}
	where $\eta$ can be chosen arbitrarily small. Since $u(x) \rightarrow \pm 1$ as 
	$x \rightarrow \pm \infty$, also $u_\varepsilon$ converges to $u$ in $L^1(A)$ and the proof of step 2
	is completed.

\item[Step 3.] 
	We can now prove the statement of the theorem observing that $\Omega$ is
	a finite union of sets of the type considered in the previous two steps, In fact, let us denote by 
	$A_1,A_2$ two sets, each given by the finite unions of 
	sets of the type considered in the previous steps and let
	$(u_\varepsilon^1, \rho_\varepsilon^1)$ (resp. $(u_\varepsilon^2,\rho_\varepsilon^2)$) 
	satisfy the stated convergence properties and energy estimates for $A_1$ (resp. $A_2$).
	We define the sequence of pairs $(u_\varepsilon, \rho_\varepsilon)$ as follows:
	\begin{align*}
		\begin{array}{rl}
		u_\varepsilon:&x\in A_1 \cup A_2 \mapsto  \left\{	\begin{array}{rl}
											u_\varepsilon^1(x), &x \in A_1\\
											u_\varepsilon^2(x), &x \in A_2.\\
											\end{array} \right. \\ 			\cr							
		\rho_\varepsilon:& x\in A_1 \cup A_2\mapsto  \left\{	\begin{array}{rl}
											\rho_\varepsilon^1(x), &x \in A_1\\
											\rho_\varepsilon^2(x), &x \in A_2.\\
											\end{array} \right. \\
		\end{array}
	\end{align*}
	It holds that $u_\varepsilon \rightarrow u$ and $\rho_\varepsilon \mathcal{L}^N \weak* \mu$ 
	on $A_1 \cup A_2$. It remains to show that
	\begin{equation*}
		\limsup_{\varepsilon \rightarrow 0^+} F_{\varepsilon}(u_\varepsilon, \rho_\varepsilon, A_1 \cup A_2) - 
		F_{\varepsilon}(u_\varepsilon, \rho_\varepsilon, A_1) 
		- F_{\varepsilon}(u_\varepsilon, \rho_\varepsilon, A_2) = 0.
	\end{equation*}
	As already observed in Step 2, we may assume that $|u_\varepsilon(x)| \leq 1$ for a.e. $x\in A_i$.
	Exploiting such a property, expanding the squares in the definition of $F_\e$, thanks to Jensen's
	inequality we obtain 
	\begin{align*}
		0& \leq F_{\varepsilon}(u_\varepsilon, \rho_\varepsilon, A_1 \cup A_2) - 
			F_{\varepsilon}(u_\varepsilon, \rho_\varepsilon, A_1) 
			- F_{\varepsilon}(u_\varepsilon, \rho_\varepsilon, A_2) \\
		&= \frac{2}{\varepsilon} \int_{A_1} \int_{A_2} J_{\varepsilon} (y-x)
			|u^2_\varepsilon(y) -u^1_\varepsilon(x)|^2 \,dy \,dx \\
		&\quad +2 \varepsilon \int_{A_1} \left( \int_{A_1} J_{\varepsilon} 
			(y-x)\frac{|u_\varepsilon^{1}(y) - u_\varepsilon^{1}(x)|}
			{\varepsilon} dy - \rho_\varepsilon^{1}(x) \right) \int_{A_2}J_{\varepsilon} (y-x)
			 \frac{|u_\varepsilon^{2}(y) - u_\varepsilon^{1}(x)|}
			{\varepsilon} \,dy \,dx \\
		&\quad + \varepsilon \int_{A_1} \left( \int_{A_2}J_{\varepsilon} (y-x) \frac{|u_\varepsilon^{2}(y) 
			- u_\varepsilon^{2}(x)|} {\varepsilon} dy  \right)^2 \,dx \\
		&\quad +2 \varepsilon \int_{A_2} \left( \int_{A_2}J_{\varepsilon} (y-x) 
			\frac{|u_\varepsilon^{2}(y) - u_\varepsilon^{2}(x)|}
			{\varepsilon} dy - \rho_\varepsilon^{2}(x) \right) \int_{A_1}J_{\varepsilon} (y-x)
			 \frac{|u_\varepsilon^{1}(y) - u_\varepsilon^{2}(x)|}
			{\varepsilon} \,dy \,dx \\
		&\quad + \varepsilon \int_{A_2} \left( \int_{A_1} J_{\varepsilon} (y-x)\frac{|u_\varepsilon^{1}(y) 
			- u_\varepsilon^{1}(x)|} {\varepsilon} dy  \right)^2 \,dx \\
		&\leq \frac{C}{\varepsilon} \int_{A_1} \int_{A_2} J_{\varepsilon} (y-x) 
			|u_\varepsilon(y) -u_\varepsilon(x)| \,dy \,dx,
	\end{align*}
	Since $u_\varepsilon$ converges to $u$ in $L^1(A_1 \cup A_2)$, we can use \cref{L2.7} and 
	Remark \ref{rm-linear-traces} to get that
	for any sequence $\varepsilon_h \rightarrow 0^+$
	the last term (with $\varepsilon$ replaced with $\varepsilon_h$) converges to zero as 
	$h  \rightarrow \infty$. The proof of the Proposition follows by the arbitrariness of $(\varepsilon_h)_h$. 
\end{itemize}\end{proof}

Using the above results we can further characterize the properties of
the limit energy density $\sigma$ as stated in the propositions below. 

\begin{prop} \label{finallyConst}
In the definition \ref{temp7} of $\mathcal{A}(e,\gamma)$ we may replace
the constraint $\int_{T_c} \rho \,dx \leq  \gamma\HN-1(C)$ by $\int_{T_c} \rho \,dx = \gamma\HN-1(C)$. 
Moreover, we may also assume that for every 
$(u,\rho) \in \mathcal{A}(e,\gamma)$ there exists a constant $C > 0$ such that
$u(x) = 1$ if $(x,e) \geq C$ and $u(x) = -1$ if $(x,e) \leq -C$.
\end{prop}
\begin{proof}
In order to simplify notation we assume that $e = e_N$.
The first assertion follows directly from the proof of \cref{limsupPolyhedral}.
For the second assertion, we notice that, again by \cref{limsupPolyhedral},
taking $Q$ a unit cube centered at zero with one face orthogonal to $e_N$, 
and $\varepsilon_h \rightarrow 0^+$, there exist sequences $(u_h)$ and $(\rho_h)$,
such that $u_h \in L^1(Q)$, $\rho_h \in L^1(Q,[0,\infty))$, $\rho_h \mathcal{L}^N 
\weak* \gamma \HN-1 \llcorner \{x_N = 0\} \text{ on } Q$, 
$(u_h)$ converges in $L^1(Q)$ to the function
\begin{equation*}
	u: x\in Q \mapsto \begin{cases} 1, &x_N > 0\\
							-1, &x_N \leq 0,
						\end{cases}
\end{equation*}
and 
\begin{equation*}
	\sigma(e_N,\gamma) \geq \limsup_{h \rightarrow \infty} F_{\varepsilon_h}(u_h,\rho_h,Q).
\end{equation*}
To conclude one can repeat the proof of \cref{liminf} and show that $\sigma(e_N,\gamma)$
can also be attained as an infimum of functions $u_h$ that equal $\pm 1$ for large values
of $|x_N|$.
\end{proof}

\begin{prop}
For any $\gamma \geq 0$, the function $\sigma( \cdot,\gamma):  S^{N-1} \rightarrow \real$ 
is upper semi-continuous.
\end{prop}
\begin{proof}
Given $\delta > 0$ and $\nu \in S^{N-1}$,
there exists an $(N-1)$-dimensional cube $C$ orthogonal to $\nu$ and a pair
$(u,\rho) \in \mathcal{A}(\nu,\gamma)$
such that  $u(x) = 1$ if $(x,e) \geq D$ and $u(x) = -1$ if $(x,e) \leq -D$ for some $D > 0$
(here we used \cref{finallyConst})
and such that $\mathcal{F}(u,\rho,T_C)\leq \sigma(\nu,\gamma) + \delta $. Thanks 
to \cref{truncation} we can also assume $|u(x)| \leq 1$ and 
$\int_{\real^N} J(h) |u(y) - u(x)| \,dy - \rho(x) \geq 0$ for a.e. $x \in T_C$. In order to simplify notation, and
 without loss of generality, in what follows we moreover assume that $\nu = e_N$ and that $C$ is of volume
 one and centered at zero. Let $(\nu_h)_h\subset S^{N-1}$ be such that $\nu_h\to e_N$ and let 
$(R_h)\subset SO(N)$ be such that $R_h e_N=\nu_h$, hence $R_h \rightarrow I$ where $I\in SO(N)$ 
denotes the identity matrix. We notice that 
$(u \circ R_h^{-1}, \rho \circ R_h^{-1}) \in \mathcal{A}(R_h(e_N),\gamma)$ and that $\mathcal{F}
(u \circ R_h^{-1},\rho \circ R_h^{-1},T_{R_h C})$ equals the functional $\mathcal{F}(u,\rho,T_C)$ 
but with $J$ replaced by $J \circ R_h$. Therefore we can estimate, using the boundedness of $u$
and the finiteness of the $L^1$-norm of $J$,
\begin{align*}
	|\mathcal{F}(u \circ R_h^{-1},\rho \circ R_h^{-1},T_{R_h C}) - \mathcal{F}(u,\rho,T_C)| 
		&\leq C \int_{T_C} \int_{\real^N} |J -  J \circ R_h|(y-x) |u(y) -u(x)| \,dy \,dx \\
	&\quad	+ C \int_{T_C} \int_{\real^N} |J -  J \circ R_h|(y-x) \rho(x) \,dy \,dx.
\end{align*}
Exploiting again $|u(x)|\leq 1$ and $u(x) = \pm 1$ if $|x_N| \geq D$ the following estimate holds true 
\begin{align*}
	\int_{T_C} \int_{\real^N} |J -  J \circ R_h|(y-x) |u(y) -u(x)| \,dy\,dx 
	&= \int_{\real^N}  |J -  J \circ R_h|(z) \int_{T_C} |u(x+z) - u(x)| \,dx\,dz \\
	&\leq  \int_{\real^N}  |J -  J \circ R_h|(z) \int_{T_C \cap \{ |x_N| < D + |z_N| \}} |u(x+z) - u(x)| \,dx\,dz \\
	& \leq C \int_{\real^N}  |J -  J \circ R_h|(z) (|z| + D) \,dz \\
	&= C \left( \int_{\real^N}  |J -  J \circ R_h|(z) \,dz + \int_{\real^N}  |J -  J \circ R_h|(z) |z|\,dz \right).
\end{align*}
This last term converges to zero as $h \rightarrow \infty$ because
$J \in L^1(\real^N)$ and the map $h\in\real^N \mapsto J(h) |h| \in L^1(\real^N)$ by (\ref{kernel}).
Moreover, it holds that
\begin{equation*}
	\int_{T_C} \int_{\real^N} |J -  J \circ R_h|(y-x) \rho(x) \,dy\,dx \leq
	|\rho|_{L^1(T_C)} \int_{\real^N}  |J -  J \circ R_h|(z) \,dz.
\end{equation*}
As a result we can write that
\begin{align*}
	\sigma(\nu,\gamma) +  \delta &\geq \mathcal{F}(u,\rho,T_C)  \\
	&\geq \limsup_{h \rightarrow \infty}
		\mathcal{F}(u \circ R_h^{-1},\rho \circ R_h^{-1},T_{R_h C}) 
\geq \limsup_{h \rightarrow \infty} \sigma(R_h(e_N),\gamma),
\end{align*}
which shows the claim by the arbitrariness of $\delta$. 
\end{proof}

Making use of the previous proposition, we now extend the proof of the limsup to the case of a
function $u\in\BV(\Omega,\{-1,1\})$ and of 
a measure $\mu$ which concentrates its $\HN-1$-absolutely continuous part on projections of finitely many
$(N-1)$ -dimensional polyhedral sets on the jump set of $u$.
\begin{prop} \label{temp8}
Let $u \in \BV(\Omega, \{-1,1\})$, $n\in\N$ and for $i\in\{1,2,\dots,n\}$ let $x_i \in \Omega \setminus S_u$,
$g$ be a piecewise constant function on the projections on $S_u$ of finitely many $(N-1)$-dimensional
polyhedral sets in $\real^N$ and let $\mu = \sum_{i = 1}^n \beta_i \delta_{x_i} + g
\HN-1 \llcorner S_u$. Then there exist functions
$u_\varepsilon \in {L^1}(\Omega) $ and  $\rho_\varepsilon \in {L^1}(\Omega,[0,\infty))$, 
such that  $u_\varepsilon \rightarrow u$ in $L^1(\Omega)$
and $\rho_\varepsilon \mathcal{L}^N \weak* \mu$  as $\varepsilon \rightarrow 0^+$ which satisfy
\begin{equation}
	\limsup _{\varepsilon \rightarrow 0^+} F_{\varepsilon}(u_\varepsilon, \rho_\varepsilon ,\Omega) 
	\leq \int_{S_u} \sigma \left( \nu_{u}, \frac{d \mu}{d \HN-1 \llcorner S_u} \right) d\HN-1.
\end{equation}
\end{prop}

\begin{proof}
Let us assume without loss of generality that 
\begin{equation*}
\mu = \beta \delta_{x_1} + \gamma \HN-1 \llcorner \Sigma,
\end{equation*}
where $\Sigma$ is the projection of $\Sigma_g \subset \real^N$, a finite union of 
$(N-1)$-dimensional polyhedral subsets, on $S_u$.
There exists (\cite{ab} section 5.4) a sequence of polyhedral functions $u_h \in \BV(\Omega, \{-1,1\})$ 
such that
\begin{equation*}
	u_h \rightarrow u \text{ in } L^1(\Omega), \quad |D u_h| \rightarrow |D u|.
\end{equation*}
If we write $\Sigma_h$ for the projection of $\Sigma_g$ to $S_{u_h}$, we have that
$\Sigma_h$ is a polyhedral subset of $S_{u_h}$ and that
\begin{equation*}
	\mu_h := \beta \delta_{x_1} + \gamma \HN-1 \llcorner \Sigma_h \weak* \mu.
\end{equation*}
We apply  \cref{limsupPolyhedral} to $u_h$ and $\mu_h$ and
use the lower semicontinuity of $\gammalimsup$  to obtain
\begin{align*}
	\gammalimsup_{\varepsilon \rightarrow 0^+}& F_{\varepsilon} (u,\mu) 
		\leq \liminf_{h \rightarrow \infty} \gammalimsup_{\varepsilon \rightarrow 0^+} 
		F_{\varepsilon} (u_h,\mu_h) \\
	& \leq \liminf_{h \rightarrow \infty} \int_{S_{u_h}} \sigma \left( \nu_{u_h}, 
		\frac{d \mu_h}{d \HN-1 \llcorner S_{u_h}} \right) \,d \HN-1 \\
	&\leq \limsup_{h \rightarrow \infty}\int_{S_{u_h} \cap \Sigma_h} \sigma (\nu_{u_h},\gamma)
		 \,d\HN-1
		+\limsup_{h \rightarrow \infty}\int_{S_{u_h} \
		\setminus \Sigma_h} \sigma (\nu_{u_h},0) \,d \HN-1.
\end{align*}
Note that for any fixed $\gamma \geq 0$, $A \subset \real^N$ open
and $v \in \BV(A, \{-1,1\})$ we know that $\sigma(\cdot, \gamma)$ is upper semi-continuous and that 
\begin{equation*}
	\int_{S_v \cap A} \sigma (\nu_{S_v},\gamma) \,d \HN-1
	= \frac{1}{2} \int_A \sigma \left( \frac{Dv}{|Dv|},\gamma \right) d |Dv|.
\end{equation*}
Approximating $\sigma(\cdot,\gamma)$ by bounded and continuous functions (e.g., via a sup-convolution
Yosida transform) and exploiting the Reshetnyak continuity theorem (see \cite{afp}) one obtains that
\begin{equation*}
	\limsup_{h \rightarrow \infty}\int_{S_{u_h} \cap \Sigma_h} \sigma (\nu_{u_h},\gamma) \,d \HN-1
	\leq \int_\Sigma \sigma (\nu_{u},\gamma) \,d \HN-1
\end{equation*}
and
\begin{equation*}
	\limsup_{h \rightarrow \infty}\int_{S_{u_h} \setminus \Sigma_h} \sigma (\nu_{u_h},0) \,d \HN-1
	\leq \int_{S_u \setminus \Sigma} \sigma (\nu_{u},0) \,d \HN-1,
\end{equation*}
i.e. the assertion.
\end{proof}

\begin{figure}
\begin{tikzpicture}
	\begin{scope}
		\clip(0,0) ellipse [x radius = 3, y radius = 2.2];
		\fill [fill = black!12] (3,-2.5) -- (-1,-3) -- (-4,1.5) -- (-1,3.5) -- cycle ;
	\end{scope}
	\begin{scope}
		\clip (-3,0) -- (0,2) -- (2,-1) -- (-1.66,-2) -- cycle;
		\draw [line width = 1.2 pt] (-3,0) .. controls (-2,-1) and (-1,2) .. (0.5,0);
		\draw [line width = 1.2 pt] (0.5,0) .. controls (1.8,-2) and (2,2) .. (2.7,-1);
	\draw [line width = 1.2 pt] (-3,0) -- (-0.5,1) -- (1,0) -- (2,0.6) -- (2.9,-0.5);
	\end{scope}
	\draw(0,0) ellipse [x radius = 3, y radius = 2.2];
	\draw (-3,0) .. controls (-2,-1) and (-1,2) .. (0.5,0);
	\draw (0.5,0) .. controls (1.8,-2) and (2,2) .. (2.7,-1);
	\draw (-3,0) -- (0,2);
	\draw (-3,0) -- (-0.5,1) -- (1,0) -- (2,0.6) -- (2.9,-0.5);
	\draw (0.4,1.6) node {$\Omega$};
	\draw (1.7,-0.7) node {$S_u$};
	\draw (-2.2,-0.4) node {$\Sigma$};
	\draw (-1.5,1.3) node {$\Sigma_g$};
	\draw (0.5,0.7) node {$\Sigma_h$};
	\draw (1.7,0.9) node {$S_{u_h}$};
\end{tikzpicture}
\caption{Sketch of the situation of Proposition \ref{temp8}}
\end{figure}

It remains to show the $\gammalimsup$-inequality in the case of a general 
measure $\mu \in \mathcal{M}(\Omega)$.

\begin{thm}[lim{-}sup inequality] \label{limsup}
Given $u \in \BV(\Omega, \{-1,1\})$ and a Radon measure $\mu \in \mathcal{M}(\Omega)$,
there are functions $u_\varepsilon \in {L^1}(\Omega) $, $u_\varepsilon \rightarrow u$ and  
$\rho_\varepsilon \in {L^1}(\Omega,[0,\infty))$, $\rho_\varepsilon \weak* \mu$
for $\varepsilon \rightarrow 0^+$, such that
\begin{equation}
	\limsup _{\varepsilon \rightarrow 0^+} F_{\varepsilon}(u_\varepsilon, \rho_\varepsilon,\Omega) 
	\leq \int_{S_u} \sigma \left( \nu_{u}, \frac{d \mu}{d \HN-1 \llcorner S_u} \right) d\HN-1.
\end{equation}
\end{thm}
\begin{proof}
Let us set
\begin{equation*}
g := \frac{d \mu}{d \HN-1 \llcorner S_u} \in L^1(S_u,\HN-1).
\end{equation*}
and notice that $g\geq 0$.
We can use the equiintegrability of $ g$ and Chebyshev's inequality to show that the functions 
$g \wedge h \in L^\infty(S_u,\HN-1)$ converge to $g$ in $L^1(S_u, \HN-1)$ as $h \rightarrow \infty$.
Moreover, since $\HN-1(S_u) < \infty$ we have that
\begin{equation*}
	\int_{S_u} g \wedge h \ d\HN-1 =  \inf \left\{ \int_{S_u} s \ d\HN-1 : s \geq g \wedge h,  
	s \text{ simple}\right\}.
\end{equation*}
Therefore, there exist simple functions $g \wedge h \leq s_h \leq h$
such that $\lim_{h \rightarrow \infty} s_h = g$ in $L^1$ and
$\HN-1(\{ s_h < g \}) < C_h$ hold, where 
$\lim_{h \rightarrow \infty} C_h = 0$.
Using the inner regularity of $\HN-1$, we can find simple functions $g_h \leq s_h$ supported on 
(finitely many) pairwise  disjoint compact subsets of $S_u$ 
that can be chosen as projections of polyhedral sets on $S_u$ such that 
$\HN-1(\{ g_h < s_h \}) < \frac{1}{h^2}$, which also implies that 
$\int_{S_u} |g_h - s_h| \leq \frac{2}{h}$. To sum up, we observe that there exist functions
\begin{equation*}
	g_h = \sum_{i=0}^{n_h} \gamma_{h,i} \chi_{K'_{h,i}},	\quad K'_{h,i} \subset S_u 
	\text{ compact, projections of polyhedral sets  on } S_u \text{, pairwise disjoint}
\end{equation*}
which satisfy
\begin{equation*}
	\lim_{h \rightarrow \infty} g_h = g \text{ in } L^1(S_u,{\mathcal H}^{N-1}), \quad
	\lim_{h \rightarrow \infty}{\mathcal H}^{N-1}(\{ g_h < g \}) = 0.
\end{equation*}
By the boundedness of $\mu$, we find measures $\sum_{i=0}^{m_h} \beta_{h,i} \delta_{x_{h,i}}$
supported outside $S_u$ such that
\begin{equation*}
	\sum_{i=0}^{m_h} \beta_{h,i} \delta_{x_{h,i}} \weak* \mu  - g {\mathcal H}^{N-1} \llcorner S_u.
\end{equation*}
It follows that
\begin{equation*}
	\mu_h := g_h {\mathcal H}^{N-1} \llcorner S_u
	+ \sum_{i=0}^{m_h} \beta_{h,i} \delta_{x_{h,i}} \weak* \mu.
\end{equation*}
We can finally combine the lower-semicontinutiy of the $\Gamma$-$\limsup$ with the result of \cref{temp8} to obtain that
\begin{align*}
	\gammalimsup_{\varepsilon \rightarrow 0^+} F_{\varepsilon}(u,\mu,\Omega)
		&\leq \liminf_{h \rightarrow \infty} \,  \gammalimsup_{\varepsilon \rightarrow 0^+}
		 F_{\varepsilon}(u,\mu_h,\Omega) \\
	&\leq \liminf_{h \rightarrow \infty} \int_{S_u} \sigma(\nu_{u},g_h) \,d {\mathcal H}^{N-1} \\
	&\leq \liminf_{h \rightarrow \infty} \int_{S_u} \sigma(\nu_{u}, g) \,d {\mathcal H}^{N-1}
		+ C\lim_{h \rightarrow \infty}{\mathcal H}^{N-1}(\{ g_h < g \}) \\
	& =\int_{S_u} \sigma(\nu_{u}, g) \,d {\mathcal H}^{N-1}.
\end{align*}
where we have exploited the fact that $\sigma$ is non-increasing in the second variable and bounded thanks to its upper-semicontinuity on $S^{N-1}$.
This concludes the proof of the theorem.
\end{proof}

\end{document}